\pdfoutput=1
\RequirePackage{ifpdf}
\ifpdf 
\documentclass[pdftex]{sigma}
\else
\documentclass{sigma}
\fi

\numberwithin{equation}{section}

\newtheorem{Theorem}{Theorem}[section]
\newtheorem{Lemma}[Theorem]{Lemma}
\newtheorem{Proposition}[Theorem]{Proposition}
 { \theoremstyle{definition}
\newtheorem{Example}[Theorem]{Example}
\newtheorem{Remark}[Theorem]{Remark} }

\begin{document}

\allowdisplaybreaks

\newcommand{\arXivNumber}{1811.07138}

\renewcommand{\PaperNumber}{032}

\FirstPageHeading

\ShortArticleName{Construction of Two Parametric Deformation of KdV-Hierarchy and its Solution}

\ArticleName{Construction of Two Parametric Deformation\\ of KdV-Hierarchy and Solution in Terms\\ of Meromorphic Functions on the Sigma Divisor\\ of a Hyperelliptic Curve of Genus 3}

\Author{Takanori AYANO~$^\dag$ and Victor M.~BUCHSTABER~$^\ddag$}

\AuthorNameForHeading{T.~Ayano and V.M.~Buchstaber}

\Address{$^\dag$~Osaka City University, Advanced Mathematical Institute, \\
\hphantom{$^\dag$}~3-3-138 Sugimoto, Sumiyoshi-ku, Osaka, 558-8585, Japan}
\EmailD{\href{mailto:ayano@sci.osaka-cu.ac.jp}{ayano@sci.osaka-cu.ac.jp}}
\URLaddressD{\url{https://researchmap.jp/ayano75/?lang=english}}

\Address{$^\ddag$~Steklov Mathematical Institute of Russian Academy of Sciences,\\
\hphantom{$^\ddag$}~8 Gubkina Street, Moscow, 119991, Russia}
\EmailD{\href{mailto:buchstab@mi-ras.ru}{buchstab@mi-ras.ru}}

\ArticleDates{Received November 21, 2018, in final form April 11, 2019; Published online April 27, 2019}

\Abstract{Buchstaber and Mikhailov introduced the polynomial dynamical systems in $\mathbb{C}^4$ with two polynomial integrals on the basis of commuting vector fields on the symmetric square of hyperelliptic curves. In our previous paper, we constructed the field of meromorphic functions on the sigma divisor of hyperelliptic curves of genus~3 and solutions of the systems for $g=3$ by these functions. In this paper, as an application of our previous results, we construct two parametric deformation of the KdV-hierarchy. This new system is integrated in the meromorphic functions on the sigma divisor of hyperelliptic curves of genus~3. In Section~8 of our previous paper [\textit{Funct. Anal. Appl.} \textbf{51} (2017), 162--176], there are miscalculations. In appendix of this paper, we correct the errors.}


\Classification{14K25; 14H40; 14H42; 14H70}

\section{Introduction}

Let $V_g$ be a hyperelliptic curve of genus $g$ defined by
\begin{gather}
V_g=\big\{(X,Y)\in\mathbb{C}^2\,|\,Y^2=X^{2g+1}\!+y_4X^{2g-1}\!-y_6X^{2g-2}\!+\cdots+y_{4g}X\!-y_{4g+2},\,y_i\in\mathbb{C}\big\}.\!\!\!\!\label{intro}
\end{gather}
A meromorphic function on the Jacobian of $V_g$ is called hyperelliptic function. The theory of hyperelliptic functions has deep relations with that of KdV-hierarchy. The KdV-hierarchy is an infinite system of differential equations defined by
\begin{gather*} U_{t_k}=\chi_kU,\qquad k=1,2,\dots,\end{gather*}
for a function $U=U(t_1,t_2,\dots)$. The functions $\chi_kU$ are determined by the recursion
\begin{gather*}\chi_{k+1}U=\mathcal{R}\chi_kU\end{gather*}
with the initial condition $\chi_1=\partial/\partial t_1$, where $\mathcal{R}$ is the Lenard operator
\begin{gather*}\mathcal{R}=\frac{1}{4}\frac{\partial^2}{\partial t_1^2}-U-\frac{1}{2}U_{t_1}\left(\frac{\partial}{\partial t_1}\right)^{-1},\end{gather*}
where $(\partial/\partial t_1)^{-1}$ implies an integral with respect to $t_1$. The KdV-equation is obtained for $k=2$
\begin{gather*}U_{t_2}=\frac{1}{4}U_{t_1t_1t_1}-\frac{3}{2}UU_{t_1}.\end{gather*}
In the theory of hyperelliptic functions associated with the model~(\ref{intro}), the hyperelliptic sigma functions play an important role. The hyperelliptic sigma functions $\sigma(w_1,w_3,\dots,w_{2g-1})$ are entire functions of $g$ complex variables, which are originally introduced by Klein as a generalization of the Weierstrass elliptic sigma functions. Baker made a significant contribution of the theory of sigma functions:
for hyperelliptic curves of genera~2 and~3, he obtained explicit expressions for higher logarithmic derivatives of sigma functions of many variables in the form of polynomials in the second and the third logarithmic derivatives of these functions \cite{Baker1,Baker2, Baker3}. Relatively recently it was shown that these differential polynomials give the fundamental equations of mathematical physics, including KdV-hierarchy and KP-equations (see \cite{B1,B2,M1}).

The surface determined by the equation $\sigma(w_1,\dots,w_{2g-1})=0$ in the Jacobian of $V_g$ is called the sigma divisor and denoted by $(\sigma)$. Let $\mathcal{F}((\sigma))$ be the field of meromorphic functions on the sigma divisor of the hyperelliptic curves of genus~3. The functions $f\in\mathcal{F}((\sigma))$ are considered as meromorphic functions on $\mathbb{C}^3$ whose restrictions to the sigma divisor $(\sigma)$ are 6-periodic. In~\cite{BM}, the polynomial dynamical systems in $\mathbb{C}^4$ with two polynomial integrals are constructed on the basis of commuting vector fields on the symmetric square of the hyperelliptic curves $V_g$. In~\cite{AB}, for $g=3$, the solutions of the systems are constructed in terms of the functions of $\mathcal{F}((\sigma))$.

For $g=2$, the dynamical systems of \cite{BM} are related to the KdV-equation~\cite{B1,BM2}. In this paper we consider the case of $g=3$ and construct two parametric deformation of the KdV-hierarchy by using the dynamical systems of~\cite{BM} (Theorem~\ref{dkdv}). We construct a solution of the new system in terms of functions of $\mathcal{F}((\sigma))$ (Theorem~\ref{dkdvs}). If $y_{12}=y_{14}=0$, then the new system goes to the system of the KdV-hierarchy in \cite[Theorem~5.2]{B1} (Proposition~\ref{dekdv}). The result of this paper is one of the applications of the results in~\cite{AB}. In \cite{M2}, an extention of the sine-Gordon equation is given and a solution is constructed in terms of the al-function on the subvariety in the hyperelliptic Jacobian. The results of this paper can be regarded as an analog of the results of~\cite{M2} for the KdV-hierarchy. In Section~\ref{section8}, we consider the rational case $(y_4,\dots,y_{14})=(0,\dots,0)$ and derive a rational solution of the KdV-hierarchy. This solution is equal to the solution obtained by the rational limit of the hyperelliptic functions of genus~2. This result would give an insight into the degeneration of the sigma functions.

In Section~8 of~\cite{AB}, we derived a solution of the dynamical systems introduced in \cite{BM} in the rational case $(y_4,\dots,y_{14})=(0,\dots,0)$ for $g=3$. Unfortunately, there are miscalculations. In Appendix~\ref{appendixA}, we correct the errors.

\section{The sigma function}

For a positive integer $g$, we set
\begin{gather*}\Delta_g=\big\{(y_4,y_6,\dots,y_{4g+2})\in\mathbb{C}^{2g}\,|\,\mbox{$Q_g(X)$ has a multiple root}\big\},\end{gather*}
where
\begin{gather*}Q_g(X)=X^{2g+1}+y_4X^{2g-1}-y_6X^{2g-2}+\cdots+y_{4g}X-y_{4g+2},\end{gather*}
and $B_g=\mathbb{C}^{2g}\backslash\Delta_g$. Consider a nonsingular hyperelliptic curve of genus $g$
\begin{gather*}
V_g=\big\{(X,Y)\in\mathbb{C}^2\,|\,Y^2=Q_g(X)\big\},
\end{gather*}
where $(y_4,y_6,\dots,y_{4g+2})\in B_g$. In this section we recall the definition of the sigma function for the curve~$V_g$ (see~\cite{B2}) and give facts about it which will be used later on. For $(X,Y)\in V_g$, let
\begin{gather*}
{\rm d}u_{2i-1}=-\frac{X^{g-i}}{2Y}{\rm d}X,\qquad 1\le i\le g,
\end{gather*}
be a basis of the vector space of holomorphic 1-forms on $V_g$, and let ${\rm d}u={}^t({\rm d}u_1,{\rm d}u_3,\dots,{\rm d}u_{2g-1})$. Further, let
\begin{gather}
{\rm d}r_{2i-1}=\frac{1}{2Y}\sum_{k=g-i+1}^{g+i-1}(-1)^{g+i-k}(k+i-g)y_{2g+2i-2k-2}X^k{\rm d}X,\qquad 1\le i\le g,\label{dr}
\end{gather}
be meromorphic one forms on $V_g$ with a pole only at $\infty$. In (\ref{dr}) we set $y_0=1$ and $y_2=0$. For example, for $g=2$
\begin{gather*}{\rm d}r_1=-\frac{X^2}{2Y}{\rm d}X,\qquad {\rm d}r_3=\frac{-y_4X-3X^3}{2Y}{\rm d}X\end{gather*}
and for $g=3$
\begin{gather*}
{\rm d}r_1=-\frac{X^3}{2Y}{\rm d}X,\qquad {\rm d}r_3=-\frac{y_4X^2+3X^4}{2Y}{\rm d}X,\\
{\rm d}r_5=-\frac{y_8X-2y_6X^2+3y_4X^3+5X^5}{2Y}{\rm d}X.
\end{gather*}
Let $\{\alpha_i,\beta_i\}_{i=1}^g$ be a canonical basis in the one-dimensional homology group of the curve $V_g$. We define the matrices of periods by
\begin{gather*}2\omega_1=\bigg(\int_{\alpha_j}{\rm d}u_{i}\bigg), \qquad 2\omega_2=\bigg(\int_{\beta_j}{\rm d}u_{i}\bigg),\qquad -2\eta_1=\bigg(\int_{\alpha_j}{\rm d}r_i\bigg), \qquad -2\eta_2=\bigg(\int_{\beta_j}{\rm d}r_i\bigg).\end{gather*}
The matrix of normalized periods has the form $\tau=\omega_1^{-1}\omega_2$. Let $\delta=\tau\delta'+\delta''$, $\delta',\delta''\in\mathbb{R}^g,$ be the vectors of Riemann's constants with respect to $(\{\alpha_i,\beta_i\},\infty)$ and $\delta:={}^t\big({}^t\delta',{}^t\delta''\big)$. Then we have $\delta'\!=\!{}^t\big(\frac{1}{2},\dots,\frac{1}{2}\big)$ and $\delta''\!=\!{}^t\big(\frac{g}{2},\frac{g-1}{2}\!,\dots,\frac{1}{2}\big)$. The sigma function $\sigma(w)$, \mbox{$w\!=\!{}^t(w_1,w_3,\dots,w_{2g-1})\!\in\!\mathbb{C}^g$}, is defined by
\begin{gather*}\sigma(w)=C\exp\left(\frac{1}{2}{}^tw\eta_1\omega_1^{-1}w\right)
\theta[\delta]\big((2\omega_1)^{-1}w,\tau\big),\end{gather*}
where $\theta[\delta](w)$ is the Riemann's theta function with characteristics $\delta$, which is defined by
\begin{gather*}\theta[\delta](w)=\sum_{n\in\mathbb{Z}^g}\exp\big\{\pi\sqrt{-1}\;{}^t(n+\delta')\tau(n+\delta')+2\pi\sqrt{-1}\;{}^t(n+\delta')(w+\delta'')\big\},\end{gather*}
and $C$ is a constant. We set $\wp_{i,j}(w)=-\partial_i\partial_j\log\sigma(w)$, $\sigma_i=\partial_i\sigma$, and $\sigma_{i,j}=\partial_i\partial_j\sigma$, where $\partial_i=\partial/\partial w_i$. We define the period lattice $\Lambda_g=\{2\omega_1m_1+2\omega_2m_2\,|\,m_1,m_2\in\mathbb{Z}^g\}$ and set $W=\{w\in\mathbb{C}^g\,|\,\sigma(w)=0\}$.

\begin{Proposition}[{\cite[Theorem 1.1]{B2} and \cite[p.~193]{N-10}}] \label{period} For $m_1,m_2\in\mathbb{Z}^g$, let $\Omega=2\omega_1m_1+2\omega_2m_2$, and let
\begin{gather*}
A=(-1)^{2({}^t\delta'm_1-{}^t\delta''m_2)+{}^tm_1m_2}\exp\big({}^t(2\eta_1m_1+ 2\eta_2m_2)(w+\omega_1m_1+\omega_2m_2)\big).
\end{gather*}
Then
\begin{enumerate}\itemsep=0pt
\item[$(i)$] $\sigma(w+\Omega)=A\sigma(w)$, where $w\in\mathbb{C}^g$,
\item[$(ii)$] $\sigma_i(w+\Omega)=A\sigma_i(w)$, $i=1,3,\dots,2g-1$, where $w\in W$.
\end{enumerate}
\end{Proposition}

Proposition \ref{period}(i) implies that $w+\Omega\in W$ for any $w\in W$ and $\Omega\in\Lambda_g$. The surface
\begin{gather*}
(\sigma):=\big\{w\in\mathbb{C}^g/\Lambda_g\,|\,\sigma(w)=0\big\}
\end{gather*}
is called the sigma divisor. We set $\deg w_{2k-1}=-(2k-1)$ and $\deg y_{2i}=2i$, where $1\le k\le g$, $2\le i\le 2g+1$. Let $S_{\mu_g}(w)$ be the Schur function associated with the partition $\mu_g=(g,g-1,\dots,1)$ and set $|\mu_g|=g+(g-1)+\cdots+1$ (see \cite[Section~4]{N-10}).

\begin{Theorem}[{\cite[Theorem~6.3]{BEL-99-R}, \cite[Theorem~7.7]{B2}, \cite{BEL-18}, \cite[Theorem 3]{N-10}}] \label{rationallim} The sigma func\-tion~$\sigma(w)$ is an entire function on $\mathbb{C}^g$, and it is given by the series
\begin{gather*}
\sigma(w)=S_{\mu_g}(w)+\sum_{i_1+3i_3+\cdots+(2g-1)i_{2g-1}>|\mu_g|}\lambda_{i_1,i_3,\dots,i_{2g-1}}w_1^{i_1}w_3^{i_3}\cdots w_{2g-1}^{i_{2g-1}},
\end{gather*}
where the coefficients $\lambda_{i_1,i_3,\dots,i_{2g-1}}\in\mathbb{Q}[y_4,y_6,\dots,y_{4g+2}]$ are homogeneous polynomials of degree $i_1+3i_3+\cdots+(2g-1)i_{2g-1}-|\mu_g|$ if $\lambda_{i_1,i_3,\dots,i_{2g-1}}\neq0$.
\end{Theorem}

\begin{Example}[{\cite[Example 4.5]{BEL-99-R}, \cite{BEL-18}, \cite[p.~192]{N-10}}]
\begin{gather*}S_{(2,1)}(w)=-w_3+\frac{1}{3}w_1^3,\qquad S_{(3,2,1)}(w)=w_1w_5-w_3^2-\frac{1}{3}w_1^3w_3+\frac{1}{45}w_1^6.\end{gather*}
\end{Example}

\section{Rational functions on the symmetric square}\label{section3}
In \cite{AB}, for $g=3$, the structure of the field of rational functions on the symmetric square of the curve $V_3$ is described explicitly. These results can be extended for any genus similarly. In this section we describe the structure of the field of rational functions on the symmetric square of the curve $V_g$.

Let $\mathcal{F}\big(V_g^2\big)$ be the field of rational functions on $V_g^2$ and let $J_g$ be the ideal in $\mathbb{C}[X_1,Y_1,X_2,Y_2]$ generated by the polynomials $Y_1^2-Q_g(X_1)$ and $Y_2^2-Q_g(X_2)$. We denote the quotient field of an integral domain $R$ by $\langle R\rangle$. We have
\begin{gather*}
\mathcal{F}\big(V_g^2\big)=\langle\mathbb{C}[X_1,Y_1,X_2,Y_2]/J_g\rangle.
\end{gather*}
Let ${\rm Sym}^2\big(\mathbb{C}^2\big)$ be the symmetric square of $\mathbb{C}^2$ and let $\mathcal{F}\big({\rm Sym}^2\big(\mathbb{C}^2\big)\big)$ be the set of rational functions $f(X_1,Y_1,X_2,Y_2)\in\mathbb{C}(X_1,Y_1,X_2,Y_2)$ such that $f(X_1,Y_1,X_2,Y_2)=f(X_2,Y_2,X_1,Y_1)$. Let ${\rm Sym}^2(V_g)$ be the symmetric square of the curve $V_g$ and let $\mathcal{F}\big({\rm Sym}^2(V_g)\big)$ be the set of elements $h\in\mathcal{F}\big(V_g^2\big)$ such that there exists a representative $\widetilde{h}\in \mathcal{F}\big({\rm Sym}^2\big(\mathbb{C}^2\big)\big)$ of~$h$. In~\cite{AB, BM}, the following elements of $\mathcal{F}\big({\rm Sym}^2\big(\mathbb{C}^2\big)\big)$ are used
\begin{gather*}
a=\frac{X_1+X_2}{2},\quad b=\frac{(X_1-X_2)^2}{4},\qquad c=\frac{Y_1-Y_2}{X_1-X_2},\qquad d=\frac{Y_1+Y_2}{2}.
\end{gather*}
Note that the elements $a$, $b$, $c$, and $d$ are algebraically independent and generate the field \linebreak $\mathcal{F}\big({\rm Sym}^2\big(\mathbb{C}^2\big)\big)$ over $\mathbb{C}$, i.e., $\mathcal{F}\big({\rm Sym}^2\big(\mathbb{C}^2\big)\big)=\mathbb{C}(a,b,c,d)$. We set
\begin{gather*}
M_g=\frac{Y_1^2-Q_g(X_1)-Y_2^2+Q_g(X_2)}{X_1-X_2},\qquad N_g=Y_1^2-Q_g(X_1)+Y_2^2-Q_g(X_2),
\end{gather*}
and $\widetilde{N}_g=-N_g/2+aM_g$. For example, for $g=2$, we obtain
\begin{gather*}
\begin{split}& M_2(a,b,c,d)=-5a^4-10a^2b-b^2+2cd-y_4\big(3a^2+b\big)+2y_6a-y_8,\\
& \widetilde{N}_2(a,b,c,d)=-4a^5+4ab^2-c^2b+2acd-d^2+2y_4\big({-}a^3+ab\big)+y_6\big(a^2-b\big)-y_{10},\end{split}
\end{gather*}
and for $g=3$, we obtain\footnote{In \cite[p.~165]{AB}, the expression of $H_{14}$ is that of $-H_{14}/2+u_2H_{12}$ in the notation of \cite{AB}.}
\begin{gather*}
M_3(a,b,c,d)=2cd-7a^6-35a^4b-21a^2b^2-b^3-y_4\big(5a^4+10a^2b+b^2\big)\\
\hphantom{M_3(a,b,c,d)=}{}+4y_6\big(a^3+ab\big)-y_8\big(3a^2+b\big)+2y_{10}a-y_{12},\\
\widetilde{N}_3(a,b,c,d)=-d^2-bc^2+2acd-6a^7-14a^5b+14a^3b^2+6ab^3\\
\hphantom{\widetilde{N}_3(a,b,c,d)=}{}-4y_4\big(a^5-ab^2\big)+y_6\big(3a^4-2a^2b-b^2\big)-2y_8\big(a^3-ab\big)+y_{10}\big(a^2-b\big)-y_{14}.
\end{gather*}
Let $A_g$ be the ideal generated by the polynomials $M_g$ and $N_g$ in the ring $\mathbb{C}[a,b,c,d]$ and let $u_i$, $i=2,4,2g-1,2g+1$, denote the elements of $\mathcal{F}\big(V_g^2\big)$ such that $u_2=a$, $u_4=b$, $u_{2g-1}=c$, and $u_{2g+1}=d$ in the field $\mathbb{C}(X_1,Y_1,X_2,Y_2)$, i.e., $u_2$, $u_4$, $u_{2g-1}$, and $u_{2g+1}$ are the equivalence classes of $a$, $b$, $c$, and $d$ in $\mathcal{F}\big(V_g^2\big)$, respectively.
Note that $u_2$, $u_4$, $u_{2g-1}$, and $u_{2g+1}$ are contained in $\mathcal{F}\big({\rm Sym}^2(V_g)\big)$. Consider the homomorphism
\begin{gather*}\Gamma_g \colon \ \mathbb{C}[a,b,c,d]\to\mathcal{F}\big({\rm Sym}^2(V_g)\big),\qquad a\mapsto u_2,\qquad b\mapsto u_4,\qquad c\mapsto u_{2g-1},\qquad d\mapsto u_{2g+1}.\end{gather*}
Then we have $\mbox{Ker}(\Gamma_g)=A_g$ and the isomorphism \cite[Lemma~3.3 and Theorem~3.4]{AB}
\begin{gather*}
\widetilde{\Gamma}_g\colon \ \langle\mathbb{C}[a,b,c,d]/A_g\rangle\to\mathcal{F}\big({\rm Sym}^2(V_g)\big).
\end{gather*}
The following two {\it commuting} derivations acting on the field $\mathcal{F}\big(V_g^2\big)$ were used in \cite{AB, BM}:
\begin{gather*}
\mathcal{L}_{2g-3}^{(g)}=\frac{1}{X_1-X_2}(\mathcal{D}_2-\mathcal{D}_1),\qquad \mathcal{L}_{2g-1}^{(g)}=\frac{1}{X_1-X_2}(X_2\mathcal{D}_1-X_1\mathcal{D}_2),
\end{gather*}
where
\begin{gather*}
\mathcal{D}_k=2Y_k\partial_{X_k}+Q_g'(X_k)\partial_{Y_k},\qquad k=1,2.
\end{gather*}
In \cite[Lemmas 16 and 17]{BM}, $\mathcal{L}_i^{(g)}u_j$ is expressed as a polynomial of $u_2$, $u_4$, $u_{2g-1}$, and $u_{2g+1}$ whose coefficients are in $\mathbb{Q}[y_4,y_6,\dots,y_{4g+2}]$ for any $i=2g-3,2g-1$ and $j=2,4,2g-1,2g+1$. These can be regarded as polynomial dynamical systems in $\mathbb{C}^4$ with coordinates $u_2$, $u_4$, $u_{2g-1}$, and~$u_{2g+1}$. We assume $g=3$.

\begin{Theorem}[{\cite[Lemmas 16 and 17]{BM}}]\label{BMsystem} In the space $\mathbb{C}^4$ with coordinates $u_2$, $u_4$, $u_5$, and~$u_7$, we have the following families of dynamical systems with constant parameters $y_4$, $y_6$, $y_8$, and~$y_{10}$:
\begin{alignat*}{3}
& (I) \quad && \mathcal{L}_3^{(3)}u_2=-u_5,\qquad \mathcal{L}_3^{(3)}u_4=-2u_7,& \\
&&& \mathcal{L}_3^{(3)}u_5=-35u_2^4-42u_2^2u_4-3u_4^2-2y_4(5u_2^2+u_4)+4y_6u_2-y_8,&\\
&&& \mathcal{L}_3^{(3)}u_7=-7\big(3u_2^5+10u_2^3u_4+3u_2u_4^2\big)-10y_4\big(u_2^3+u_2u_4\big)& \\
&&& \hphantom{\mathcal{L}_3^{(3)}u_7=}{} +2y_6\big(3u_2^2+u_4\big)-3y_8u_2+y_{10},\\
&(II) \quad && \mathcal{L}_5^{(3)}u_2=u_2u_5-u_7,\qquad \mathcal{L}_5^{(3)}u_4=2(u_2u_7-u_4u_5),& \\
&&& \mathcal{L}_5^{(3)}u_5=u_5^2+14u_2^5-28u_2^3u_4-18u_2u_4^2-8y_4u_2u_4+2y_6\big(u_2^2+u_4\big)-2y_8u_2+y_{10},&\\
&&& \mathcal{L}_5^{(3)}u_7=-u_5u_7+21u_2^6+35u_2^4u_4-21u_2^2u_4^2-3u_4^3+2y_4\big(5u_2^4-u_4^2\big)& \\
&&& \hphantom{\mathcal{L}_5^{(3)}u_7=}{} -2y_6\big(3u_2^3-u_2u_4\big)+y_8\big(3u_2^2-u_4\big)-y_{10}u_2.&
\end{alignat*}
\end{Theorem}

The systems (I) and (II) have common first integrals $H_{12}:=M_3(u_2,u_4,u_5,u_7)+y_{12}$ and $H_{14}:=\widetilde{N}_3(u_2,u_4,u_5,u_7)+y_{14}$ \cite{BM,BM2}, \cite[Theorem 7.1]{AB}. Moreover, the system (I) is a Hamiltonian system with the Hamiltonian $H_{12}$ and the Poisson structure determined by $\{u_2,u_7\}=-1/2$, $\{u_4,u_5\}=-1$, $\{u_2,u_4\}=\{u_2,u_5\}=\{u_4,u_7\}=\{u_5,u_7\}=0$~\cite{BM2}. The system (II) is a Hamiltonian system with the Hamiltonian $H_{14}$ and the Poisson structure determined by $\{u_2,u_7\}=1/2$, $\{u_4,u_5\}=1$, $\{u_2,u_4\}=\{u_2,u_5\}=\{u_4,u_7\}=\{u_5,u_7\}=0$~\cite{BM2}. These Hamiltonians are in involution with respect to the Poisson structures and the systems are Liouville integrable~\cite{BM2}.

\section{Meromorphic functions on the sigma divisor}\label{section4}

In \cite{AB}, for $g=3$, the field of meromorphic functions on the sigma divisor of the curve $V_3$ is described. In this section we recall these results.

We assume $g=3$. Fix any constant vector $(y_4,y_6,y_8,y_{10},y_{12},y_{14})\in B_3$. Let $\mathcal{F}$ be the field of all meromorphic functions on $\mathbb{C}^3$ and let $\mathcal{F}[(\sigma)]$ be the set of meromorphic functions $f\in\mathcal{F}$ satisfying the following two conditions:
\begin{itemize}\itemsep=0pt
\item for any point $w\in W$, there exist an open neighborhood $U_1\subset\mathbb{C}^3$ of this point and two holomorphic functions $g$ and $h$ on $U_1$ such that the function~$h$ does not identically vanish on $U_1\cap W$ and $f=g/h$ on $U_1$;
\item $f(w+\Omega)=f(w)$ for any $w\in W$ and $\Omega\in\Lambda_3$.
\end{itemize}
Note that $\mathcal{F}[(\sigma)]$ is a subring in $\mathcal{F}$, but it is not generally a field. Let us consider the Abel--Jacobi map
\begin{gather*}
I_3\colon \ {\rm Sym}^2(V_3)\to{\rm Jac}(V_3)=\mathbb{C}^3/\Lambda_3,\qquad (P_1,P_2)\mapsto\int_{\infty}^{P_1}{\rm d}u+\int_{\infty}^{P_2}{\rm d}u.
\end{gather*}
The Abel--Jacobi map $I_3$ induces a ring homomorphism
\begin{gather*}
I_3^* \colon \ \mathcal{F}[(\sigma)]\to\mathcal{F}\big({\rm Sym}^2(V_3)\big),\qquad f\mapsto f\circ I_3.
\end{gather*}
Let $J^*$ be the set of meromorphic functions $f\in\mathcal{F}[(\sigma)]$ identically vanishing on $W$. Thus, we have $\operatorname{Ker} I_3^*=J^*$. We set $\mathcal{F}((\sigma))=\mathcal{F}[(\sigma)]/J^*$. Then $\mathcal{F}((\sigma))$ is a field and, by construction, there is an isomorphism of fields (see \cite[Section~4]{AB})
\begin{gather*}\overline{I_3^*} \colon \ \mathcal{F}((\sigma)) \to \mathcal{F}\big({\rm Sym}^2(V_3)\big).\end{gather*}
The following meromorphic functions on $\mathbb{C}^3$ are introduced in \cite{AB}:
\begin{gather*}
f_1=\frac{\sigma_{1,1}}{\sigma_1},\qquad f_2=\frac{\sigma_3}{\sigma_1},\qquad
f_3=\frac{\sigma_{1,3}}{\sigma_1},\qquad f_4=\frac{\sigma_5}{\sigma_1},\\
f_5=\frac{\sigma_{3,3}}{\sigma_1},\qquad g_5=\frac{\sigma_{1,5}}{\sigma_1},\qquad f_7=\frac{\sigma_{3,5}}{\sigma_1},\\
F_2=-\frac{1}{2}f_2,\qquad F_4=\frac{1}{4}f_2^2-f_4,\qquad F_5=\frac{1}{2}(f_1f_2^2+f_5-2f_2f_3),\\
F_7=\frac{1}{4}(2f_2^2f_3-2f_3f_4-f_1f_2^3+2f_1f_2f_4-f_2f_5+2f_7-2f_2g_5).
\end{gather*}
We have $F_i\in\mathcal{F}[(\sigma)]$ and $I_3^*(F_i)=u_i$ for $i=2,4,5,7$ (see \cite[Proposition~4.1]{AB}). In \cite{AB}, the following derivations acting on the field $\mathcal{F}((\sigma))$ are introduced
\begin{gather*}
L_3^{(3)}=\partial_3-\frac{\sigma_3}{\sigma_1}\partial_1,\qquad L_5^{(3)}=\partial_5-\frac{\sigma_5}{\sigma_1}\partial_1.
\end{gather*}
\begin{Lemma}[{\cite[Lemma 6.4]{AB}}]\label{16} The relations $\mathcal{L}^{(3)}_3\circ\overline{I_3^*}=\overline{I_3^*}\circ L_3^{(3)}$ and $\mathcal{L}^{(3)}_5\circ\overline{I_3^*}=\overline{I_3^*}\circ L_5^{(3)}$ hold.
\end{Lemma}

\section{Two parametric deformation of KdV-hierarchy}

We assume $g=3$ and consider the following derivations:
\begin{gather*}T_1:=\partial_1-\frac{\sigma_1}{\sigma_5}\partial_5=-f_4^{-1}L_5^{(3)},\qquad T_3:=\partial_3-\frac{\sigma_3}{\sigma_5}\partial_5=L_3^{(3)}-f_2f_4^{-1}L_5^{(3)}.\end{gather*}
From \cite[Lemma 6.2]{AB}, the commutation relation $[T_1,T_3]=0$ holds in the Lie algebra of derivations of~$\mathcal{F}$. Since the operators $L_3^{(3)}$ and $L_5^{(3)}$ are the derivations of the field $\mathcal{F}((\sigma))$ and $f_2,f_4^{-1}\in \mathcal{F}[(\sigma)]$, the operators~$T_1$ and~$T_3$ are also the derivations of the field~$\mathcal{F}((\sigma))$. We consider the following derivations acting on the field~$\mathcal{F}\big(V_g^2\big)$
\begin{gather}
\mathcal{T}_1=-\frac{1}{X_1X_2}\mathcal{L}^{(3)}_5=-\frac{1}{X_1X_2(X_1-X_2)}(X_2\mathcal{D}_1-X_1\mathcal{D}_2),\label{newderia1} \\
\mathcal{T}_3=\mathcal{L}^{(3)}_3+\frac{X_1+X_2}{X_1X_2}\mathcal{L}^{(3)}_5=\frac{1}{X_1-X_2}(\mathcal{D}_2-\mathcal{D}_1)+\frac{X_1+X_2}{X_1X_2(X_1-X_2)}(X_2\mathcal{D}_1-X_1\mathcal{D}_2).\label{newderia2}
\end{gather}

\begin{Proposition}\label{commutealg} The commutation relation $[\mathcal{T}_1,\mathcal{T}_3]=0$ holds.
\end{Proposition}

\begin{proof}Since $\big[\mathcal{L}^{(3)}_3,\mathcal{L}^{(3)}_5\big]=0$, the direct calculation shows the proposition.
\end{proof}

\begin{Proposition}\label{newalg}We have
\begin{gather*}
\mathcal{T}_1X_1=\frac{-2Y_1}{X_1(X_1-X_2)},\qquad \mathcal{T}_1Y_1=\frac{-Q_3'(X_1)}{X_1(X_1-X_2)},\\
\mathcal{T}_1X_2=\frac{2Y_2}{X_2(X_1-X_2)},\qquad \mathcal{T}_1Y_2=\frac{Q_3'(X_2)}{X_2(X_1-X_2)},\\
\mathcal{T}_3X_1=\frac{2X_2Y_1}{X_1(X_1-X_2)},\qquad \mathcal{T}_3Y_1=\frac{X_2Q_3'(X_1)}{X_1(X_1-X_2)},\\
\mathcal{T}_3X_2=\frac{-2X_1Y_2}{X_2(X_1-X_2)},\qquad \mathcal{T}_3Y_2=\frac{-X_1Q_3'(X_2)}{X_2(X_1-X_2)}.
\end{gather*}
\end{Proposition}

\begin{proof}The direct calculation shows the proposition.
\end{proof}

\begin{Lemma}\label{17}The relations $\mathcal{T}_1\circ\overline{I_3^*}=\overline{I_3^*}\circ T_1$ and $\mathcal{T}_3\circ\overline{I_3^*}=\overline{I_3^*}\circ T_3$ hold.
\end{Lemma}

\begin{proof}From Lemma \ref{16}, $I_3^*(f_2)=-(X_1+X_2)$, and $I_3^*(f_4)=X_1X_2$, we obtain the lemma.
\end{proof}

\begin{Proposition}\label{26} In the space $\mathbb{C}^4$ with coordinates $u_2$, $u_4$, $u_5$, and~$u_7$, we have the following families of rational dynamical systems with constant parameters $y_4$, $y_6$, $y_8$, and~$y_{10}$:
\begin{gather*}
\mathcal{T}_1u_2=\frac{u_2u_5-u_7}{u_4-u_2^2},\qquad \mathcal{T}_1u_4=\frac{2(u_2u_7-u_4u_5)}{u_4-u_2^2},\\
\mathcal{T}_1u_5=\big(u_4-u_2^2\big)^{-1}\big\{u_5^2+14u_2^5-28u_2^3u_4-18u_2u_4^2-8y_4u_2u_4\\
\hphantom{\mathcal{T}_1u_5=}{} +2y_6\big(u_2^2+u_4\big)-2y_8u_2+y_{10}\big\},\\
\mathcal{T}_1u_7=\big(u_4-u_2^2\big)^{-1}\big\{{-}u_5u_7+21u_2^6+35u_2^4u_4-21u_2^2u_4^2-3u_4^3+2y_4\big(5u_2^4-u_4^2\big)\\
\hphantom{\mathcal{T}_1u_7=}{} -2y_6\big(3u_2^3-u_2u_4\big)+y_8\big(3u_2^2-u_4\big)-y_{10}u_2\big\},\\
\mathcal{T}_3u_2=\frac{2u_2u_7-u_4u_5-u_2^2u_5}{u_4-u_2^2},\qquad \mathcal{T}_3u_4=\frac{2\big(2u_2u_4u_5-u_4u_7-u_2^2u_7\big)}{u_4-u_2^2},\\
\mathcal{T}_3u_5=\big(u_4-u_2^2\big)^{-1}\big\{{-}2u_2u_5^2+7u_2^6+63u_2^4u_4-3u_2^2u_4^2-3u_4^3+2y_4\big(5u_2^4+4u_2^2u_4-u_4^2\big)\\
\hphantom{\mathcal{T}_3u_5=}{} -8y_6u_2^3+y_8\big(5u_2^2-u_4\big)-2y_{10}u_2\big\},\\
\mathcal{T}_3u_7=\big(u_4-u_2^2\big)^{-1}\big\{2u_2u_5u_7-21u_2^7-21u_2^5u_4-7u_2^3u_4^2-15u_2u_4^3-2y_4\big(5u_2^5+3u_2u_4^2\big)\\
\hphantom{\mathcal{T}_3u_7=}{} +2y_6\big(3u_2^4+u_4^2\big)-y_8\big(3u_2^3+u_2u_4\big)+y_{10}\big(u_2^2+u_4\big)\big\}.
\end{gather*}
\end{Proposition}

\begin{proof}From Theorem \ref{BMsystem}, (\ref{newderia1}), and (\ref{newderia2}), we obtain the proposition.
\end{proof}

Let $u=4u_2$ and $v=2\big(u_4-u_2^2\big)$. For any $w\in\mathcal{F}\big({\rm Sym}^2(V_3)\big)$, we use the notation $w'=\mathcal{T}_1w$ and $\dot{w}=\mathcal{T}_3w$. Then we obtain the main result of this paper.

\begin{Theorem}\label{dkdv}We obtain the following new system that can be called two parametric deformed KdV-hierarchy:
\begin{gather}
v^4(u'''-4\dot{u}-6uu')-32y_{12}v\dot{u}+32y_{14}(vu'-3u\dot{u})=0,\label{first} \\
v^4(\dot{u}''-4u\dot{u}-2u'v)-32y_{12}v\dot{v}+32y_{14}(vv'-3u\dot{v})=0,\label{second} \\
\dot{u}=v',\label{third} \\
2\dot{v}=vu'-uv'.\label{fourth}
\end{gather}
\end{Theorem}

\begin{proof}From Proposition \ref{26}, the direct calculation shows
\begin{gather}
u_2''=2u_4+10u_2^2+y_4-\frac{y_{12}}{\big(u_4-u_2^2\big)^2}-\frac{4y_{14}u_2}{\big(u_4-u_2^2\big)^3},\label{seconddif}
\end{gather}
where in the above calculation we deleted the terms $u_5u_7$ and $u_7^2$ by using the relations
\begin{gather*} M_3(u_2,u_4,u_5,u_7)=\widetilde{N}_3(u_2,u_4,u_5,u_7)=0\end{gather*} in $\mathcal{F}\big({\rm Sym}^2(V_3)\big)$ (see Section~\ref{section3}). By differentiating the both sides of~(\ref{seconddif}) with respect to~$\mathcal{T}_1$, we obtain
\begin{gather*}
u_2''' = -\frac{4\big(4u_2u_7 + u_4u_5 - 5u_2^2u_5\big)}{u_4-u_2^2}+4y_{12}\frac{2u_2u_7 - u_4u_5 - u_2^2u_5}{\big(u_4-u_2^2\big)^4}\\
\hphantom{u_2''' =}{} +4y_{14}\frac{u_4u_7+11u_2^2u_7 - 7u_2u_4u_5 - 5u_2^3u_5}{\big(u_4-u_2^2\big)^5}\\
\hphantom{u_2'''}{} =4\dot{u}_2+24u_2u_2'+\frac{4y_{12}\dot{u}_2}{\big(u_4-u_2^2\big)^3}+\frac{4y_{14}\{\big(u_2^2-u_4\big)u_2'+6u_2\dot{u}_2\}}{\big(u_4-u_2^2\big)^4}.
\end{gather*}
Therefore we obtain the equation (\ref{first}). By differentiating the both sides of (\ref{seconddif}) with respect to $\mathcal{T}_3$, we obtain
\begin{gather*}
\dot{u}_2''=-4\frac{u_7\big(u_4-9u_2^2\big)+u_2u_5\big(3u_4+5u_2^2\big)}{u_4-u_2^2}+4y_{12}\frac{u_2u_5\big(3u_4+u_2^2\big)-u_7\big(u_4+3u_2^2\big)}{\big(u_4-u_2^2\big)^4} \\
\qquad {} +4y_{14}\frac{u_5\big(u_4^2+18u_2^2u_4+5u_2^4\big)-8u_2u_7\big(u_4+2u_2^2\big)}{\big(u_4-u_2^2\big)^5}, \\
\quad {} =16u_2\dot{u}_2+4\big(u_4-u_2^2\big)u_2'+2y_{12}\frac{\dot{\big(u_4-u_2^2\big)}}{\big(u_4-u_2^2\big)^3}-2y_{14}\frac{6u_2\dot{\big(u_2^2-u_4\big)}+\big(u_2^2-u_4\big)\big(u_2^2-u_4\big)'}{\big(u_4-u_2^2\big)^4}.
\end{gather*}
Therefore we obtain the equation (\ref{second}). From Proposition~\ref{26}, we obtain the equations (\ref{third}) and~(\ref{fourth}).
\end{proof}

\section{Relation with a curve of genus 2}

Let us consider the homomorphism of the field of rational functions $\mathbb{C}(X_1,Y_1,X_2,Y_2)$
\begin{gather*}\psi\colon \ \mathbb{C}(X_1,Y_1,X_2,Y_2)\to\mathbb{C}(X_1,Y_1,X_2,Y_2),\qquad X_i\mapsto X_i,\qquad Y_i\mapsto \frac{Y_i}{X_i},\qquad i=1,2.\end{gather*}
The map $\psi$ induces the homomorphism
\begin{gather*}{\rm Sym}(\psi)\colon \ \mathcal{F}\big({\rm Sym}^2\big(\mathbb{C}^2\big)\big)\to\mathcal{F}\big({\rm Sym}^2\big(\mathbb{C}^2\big)\big).\end{gather*}
In Section~\ref{section3}, we noted $\mathcal{F}\big({\rm Sym}^2\big(\mathbb{C}^2\big)\big)=\mathbb{C}(a,b,c,d)$. The map ${\rm Sym}(\psi)$ transforms the generators $a$, $b$, $c$, and $d$ as follows
\begin{gather}
a\mapsto a,\qquad b \mapsto b,\qquad c\mapsto \frac{ac-d}{a^2-b},\qquad d\mapsto \frac{ad-bc}{a^2-b}.\label{trans}
\end{gather}
Fix any constant vector $(y_4,y_6,y_8,y_{10})\in\mathbb{C}^4$. We consider a curve $V_2$
\begin{gather*}V_2=\big\{(X,Y)\in\mathbb{C}^2\,|\,Y^2=X^5+y_4X^3-y_6X^2+y_8X-y_{10}\big\}\end{gather*}
and a curve $V_{3,2}$
\begin{gather*}V_{3,2}=\big\{(X,Y)\in\mathbb{C}^2\,|\,Y^2=X^7+y_4X^5-y_6X^4+y_8X^3-y_{10}X^2\big\}.\end{gather*}
The map $\psi$ induces the homomorphism
\begin{gather*}\psi_1\colon \ \mathcal{F}\big(V_2^2\big)\to\mathcal{F}\big(V_{3,2}^2\big),\qquad X_i\mapsto X_i,\qquad Y_i\mapsto \frac{Y_i}{X_i},\qquad i=1,2.\end{gather*}

\begin{Proposition}The map $\psi_1$ is an isomorphism between $\mathcal{F}\big(V_2^2\big)$ and $\mathcal{F}\big(V_{3,2}^2\big)$.
\end{Proposition}

\begin{proof}We can consider the map
\begin{gather*}\psi_2 \colon \ \mathcal{F}\big(V_{3,2}^2\big)\to\mathcal{F}\big(V_2^2\big),\qquad X_i\to X_i,\qquad Y_i\to X_iY_i,\qquad i=1,2.\end{gather*}
Then we can check
\begin{gather*}\psi_2\circ\psi_1=\operatorname{id}_{\mathcal{F}\big(V_2^2\big)},\qquad \psi_1\circ\psi_2=\operatorname{id}_{\mathcal{F}\big(V_{3,2}^2\big)}.\tag*{\qed}\end{gather*}\renewcommand{\qed}{}
\end{proof}

\begin{Proposition}The map $\psi_1$ is an isomorphism between $\mathcal{F}\big({\rm Sym}^2(V_2)\big)$ and $\mathcal{F}({\rm Sym}^2(V_{3,2}))$, and we have
\begin{gather}
\psi_1(u_2)=u_2,\qquad\psi_1(u_4)=u_4,\qquad\psi_1(u_3)=\frac{u_2u_5-u_7}{u_2^2-u_4},\qquad\psi_1(u_5)=\frac{u_2u_7-u_4u_5}{u_2^2-u_4}.\label{trans2}
\end{gather}
\end{Proposition}

\begin{proof}Since $\psi_1\big(\mathcal{F}\big({\rm Sym}^2(V_2)\big)\big)\subset\mathcal{F}\big({\rm Sym}^2(V_{3,2})\big)$ and $\psi_2\big(\mathcal{F}\big({\rm Sym}^2(V_{3,2})\big)\big)\subset\mathcal{F}\big({\rm Sym}^2(V_2)\big)$, the map $\psi_1$ is an isomorphism between $\mathcal{F}\big({\rm Sym}^2(V_2)\big)$ and $\mathcal{F}\big({\rm Sym}^2(V_{3,2})\big)$. From~(\ref{trans}) we obtain the relations~(\ref{trans2}).
\end{proof}

\begin{Proposition}\label{tr}We have
\begin{gather*}\mathcal{T}_1\circ\psi_1=\psi_1\circ\mathcal{L}_1^{(2)},\qquad \mathcal{T}_3\circ\psi_1=\psi_1\circ\mathcal{L}_3^{(2)}.\end{gather*}
\end{Proposition}

\begin{proof}By the direct calculation we can check $\mathcal{T}_1\circ\psi_1(X_i)=\psi_1\circ\mathcal{L}_1^{(2)}(X_i)$ and $\mathcal{T}_1\circ\psi_1(Y_i)=\psi_1\circ\mathcal{L}_1^{(2)}(Y_i)$ for $i=1,2$.
Therefore we obtain $\mathcal{T}_1\circ\psi_1=\psi_1\circ\mathcal{L}_1^{(2)}$.
Similary, we obtain $\mathcal{T}_3\circ\psi_1=\psi_1\circ\mathcal{L}_3^{(2)}$.
\end{proof}

We assume $(y_4,y_6,y_8,y_{10})\in B_2$. Let us consider the Abel--Jacobi map of the curve $V_2$
\begin{gather*}
I_2 \colon \ {\rm Sym}^2(V_2)\to{\rm Jac}(V_2)=\mathbb{C}^2/\Lambda_2,\qquad (P_1,P_2)\mapsto\int_{\infty}^{P_1}{\rm d}u+\int_{\infty}^{P_2}{\rm d}u.
\end{gather*}
Let $\mathcal{F}({\rm Jac}(V_2))$ be the field of meromorphic functions on the Jacobian ${\rm Jac}(V_2)$.
The Abel--Jacobi map $I_2$ induces the isomorphism of the fields:
\begin{gather*}
I_2^* \colon\ \mathcal{F}({\rm Jac}(V_2))\to\mathcal{F}\big({\rm Sym}^2(V_2)\big),\qquad f\mapsto f\circ I_2.
\end{gather*}
As derivations of $\mathcal{F}\big(V_2^2\big)$, the derivations $\mathcal{L}_1^{(2)}$ and $\mathcal{L}_3^{(2)}$ can be expressed as \cite[Section~6]{AB}\footnote{In~\cite{AB}, these expressions are given for $g=3$ and we can prove them for any~$g$ similarly.}
\begin{gather*}
\mathcal{L}_1^{(2)}=\frac{1}{X_1-X_2}\{-2Y_1({\rm d}_{P_1}/{\rm d}X_1)+2Y_2({\rm d}_{P_2}/{\rm d}X_2)\},\\
\mathcal{L}_3^{(2)}=\frac{1}{X_1-X_2}\{2X_2Y_1({\rm d}_{P_1}/{\rm d}X_1)-2X_1Y_2({\rm d}_{P_2}/{\rm d}X_2)\},
\end{gather*}
where $P_i=(X_i,Y_i)\in V_2$, $i=1,2$, we regard $X_i$ and $Y_i$ as meromorphic functions on~$V_2^2$, and~${\rm d}X_i$ and~${\rm d}Y_i$ are the total differentials of $X_i$ and~$Y_i$ for $i=1,2$. Let us describe the action of these operators in more detail. For $g(P_1,P_2)\in\mathcal{F}\big(V_2^2\big)$, ${\rm d}_{P_i}(g)$ is the total differential of~$g$ as a~meromorphic function of~$P_i$. Then ${\rm d}_{P_i}(g)/{\rm d}X_i$ is the meromorphic function on $V_2^2$ determined uniquely by ${\rm d}_{P_i}(g)=({\rm d}_{P_i}(g)/{\rm d}X_i)\cdot {\rm d}X_i$. We consider the following derivations of $\mathcal{F}({\rm Jac}(V_2))$
\begin{gather*}L_1^{(2)}=\frac{\partial}{\partial w_1},\qquad L_3^{(2)}=\frac{\partial}{\partial w_3}.\end{gather*}

\begin{Lemma}We have $\mathcal{L}_1^{(2)}\circ I_2^*=I_2^*\circ L_1^{(2)}$ and $\mathcal{L}_3^{(2)}\circ I_2^*=I_2^*\circ L_3^{(2)}$.
\end{Lemma}

\begin{proof}Set $h\in\mathcal{F}({\rm Jac}(V_2))$ and $w=I_2((P_1,P_2))$. We have
\begin{gather*}
\mathcal{L}_1^{(2)}\circ I_2^*(h)=\mathcal{L}_1^{(2)}(h(w))=\frac{1}{X_1-X_2}\{-2Y_1({\rm d}_{P_1}(h(w))/{\rm d}X_1)+2Y_2({\rm d}_{P_2}(h(w))/{\rm d}X_2)\} \\
\qquad {} =\frac{1}{X_1-X_2}\left\{-2Y_1\left(-\frac{X_1}{2Y_1}h_1(w)-\frac{1}{2Y_1}h_3(w)\right) +2Y_2\left(-\frac{X_2}{2Y_2}h_1(w)-\frac{1}{2Y_2}h_3(w)\right)\right\} \\
\qquad {} =h_1(w)=I_2^*\circ L_1^{(2)}(h),
\end{gather*}
where $h_i=\partial_{w_i}h$. The lemma's assertions for the operator $L_3^{(2)}$ are proved similarly.
\end{proof}

By the isomorphism
\begin{gather*}
I_2^* \colon \ \mathcal{F}({\rm Jac}(V_2))\simeq\mathcal{F}\big({\rm Sym}^2(V_2)\big),
\end{gather*}
the operators $L_1^{(2)}$ and $L_3^{(2)}$ transform into $\mathcal{L}_1^{(2)}$ and $\mathcal{L}_3^{(2)}$, respectively.
By the isomorphism
\begin{gather*}\psi_1\colon \ \mathcal{F}\big({\rm Sym}^2(V_2)\big)\simeq\mathcal{F}\big({\rm Sym}^2(V_{3,2})\big),\end{gather*}
the operators $\mathcal{L}_1^{(2)}$ and $\mathcal{L}_3^{(2)}$ transform into the operators $\mathcal{T}_1$ and $\mathcal{T}_3$, respectively.

\begin{Proposition}\label{dekdv}If $y_{12}=y_{14}=0$ and $v\neq0$, the system \eqref{first}, \eqref{second}, and \eqref{third} in Theorem~{\rm \ref{dkdv}} goes to the system of the KdV-hierarchy in {\rm \cite[Theorem 5.2]{B1}}.
\end{Proposition}

\begin{proof}
If $y_{12}=y_{14}=0$ and $v\neq0$, the equation (\ref{first}) goes to
\begin{gather*}u'''-4\dot{u}-6uu'=0.\end{gather*}
Therefore we obtain
\begin{gather*}u'''=3(u^2)'+4\dot{u}.\end{gather*}
On the other hand, the equation (\ref{second}) goes to
\begin{gather*}\dot{u}''-4u\dot{u}-2u'v=0.\end{gather*}
From (\ref{third}), the above equation becomes
\begin{gather*}v'''-4uv'-2u'v=0.\end{gather*}
From (\ref{fourth}), we obtain
\begin{gather*}0=v'''-4uv'-2u'v=v'''-3uv'+2\dot{v}-vu'-2u'v=v'''-3uv'-3u'v+2\dot{v}.\end{gather*}
Thus we have
\begin{gather*}v'''=3(uv)'-2\dot{v}.\tag*{\qed}\end{gather*}\renewcommand{\qed}{}
\end{proof}

\section{Solution of the two parametric deformed KdV-hierarchy}

We assume $g=3$. The two parametric deformed KdV-hierarchy introduced in Theorem~\ref{dkdv} is integrated in functions of $\mathcal{F}((\sigma))$. Consider a constant vector $(y_4,y_6,y_8,y_{10},y_{12},y_{14})\in B_3$.
Take a point $w^{(0)}=\big(w_1^{(0)},w_3^{(0)},w_5^{(0)}\big)\in W$ such that $\sigma_i\big(w^{(0)}\big)\neq0$ for $i=1,5$. In a sufficiently small open neighborhood $U_2\subset\mathbb{C}^2$ of $\big(w_1^{(0)},w_3^{(0)}\big)$, there exists a uniquely determined holomorphic function $\xi(w_1,w_3)$ on $U_2$ such that $\xi\big(w_1^{(0)},w_3^{(0)}\big)=w_5^{(0)}$, $(w_1,w_3,\xi(w_1,w_3))\in W$ for any point $(w_1,w_3)\in U_2$, and $\sigma_i(w_1,w_3,\xi(w_1,w_3))\neq0$ for any point $(w_1,w_3)\in U_2$ and $i=1,5$.
\begin{Lemma}\label{36} For any $F\in\mathcal{F}$, we have
\begin{gather*}\frac{\partial}{\partial w_1}F(w_1,w_3,\xi(w_1,w_3))=T_1(F),\qquad \frac{\partial}{\partial w_3}F(w_1,w_3,\xi(w_1,w_3))=T_3(F).\end{gather*}
\end{Lemma}

\begin{proof}According to the definition of the function $\xi$, we have
\begin{gather*}
\frac{\partial \xi}{\partial w_1}=-\frac{\sigma_1}{\sigma_5}(w_1,w_3,\xi(w_1,w_3)),\qquad \frac{\partial \xi}{\partial w_3}=-\frac{\sigma_3}{\sigma_5}(w_1,w_3,\xi(w_1,w_3)).
\end{gather*}
Therefore
\begin{gather*}
\frac{\partial}{\partial w_1}F(w_1,w_3,\xi(w_1,w_3))=\partial_1F-\frac{\sigma_1}{\sigma_5}(\partial_5F)=T_1(F),\\
\frac{\partial}{\partial w_3}F(w_1,w_3,\xi(w_1,w_3))=\partial_3F-\frac{\sigma_3}{\sigma_5}(\partial_5F)=T_3(F).\tag*{\qed}
\end{gather*}\renewcommand{\qed}{}
\end{proof}

We set $U(x,t)=4F_2(x,t,\xi(x,t))$ and $V(x,t)=2\big\{F_4(x,t,\xi(x,t))-F_2(x,t,\xi(x,t))^2\big\}$. For a~function $K(x,t)$, we use the notation $K'=\partial_x K,\;\dot{K}=\partial_t K$.

\begin{Theorem}\label{dkdvs}The functions $U$ and $V$ satisfy the two parametric deformed KdV-hierarchy
\begin{gather*}
V^4(U'''-4\dot{U}-6UU')-32y_{12}V\dot{U}+32y_{14}(VU'-3U\dot{U})=0,\\
V^4(\dot{U}''-4U\dot{U}-2U'V)-32y_{12}V\dot{V}+32y_{14}(VV'-3U\dot{V})=0,\\
\dot{U}=V',\\
2\dot{V}=VU'-UV'.
\end{gather*}
\end{Theorem}

\begin{proof}From Lemmas \ref{17}, \ref{36}, and Theorem \ref{dkdv}, we obtain the theorem.
\end{proof}

\section{The rational limit}\label{section8}

Let the constant vector $(y_4,\dots,y_{14})\in\mathbb{C}^6$ tend to zero. Then, according to Theorem~\ref{rationallim}, the sigma function $\sigma(w_1,w_3,w_5)$ transforms into the Schur--Weierstrass polynomial (see~\cite{BEL-99-R})
\begin{gather*}
\sigma=w_1w_5-w_3^2-\frac{1}{3}w_1^3w_3+\frac{1}{45}w_1^6.
\end{gather*}
As a result, we obtain
\begin{gather*}
\sigma_1=w_5-w_1^2w_3+\frac{2}{15}w_1^5,\qquad \sigma_3=-2w_3-\frac{1}{3}w_1^3,\qquad \sigma_5=w_1,\\
\sigma_{11}=-2w_1w_3+\frac{2}{3}w_1^4,\qquad \sigma_{13}=-w_1^2,\qquad \sigma_{15}=1, \qquad \sigma_{33}=-2,\qquad \sigma_{35}=0.
\end{gather*}
Take a point $w^{(0)}=\big(w_1^{(0)},w_3^{(0)},w_5^{(0)}\big)\in W$ such that $\sigma_i\big(w^{(0)}\big)\neq0$ for $i=1,5$. In a sufficiently small open neighborhood $U_2\subset\mathbb{C}^2$ of $\big(w_1^{(0)},w_3^{(0)}\big)$, there exists a uniquely determined holomorphic function $\xi(w_1,w_3)$ on $U_2$ such that $\xi\big(w_1^{(0)},w_3^{(0)}\big)=w_5^{(0)}$, $\big(w_1,w_3,\xi(w_1,w_3)\big)\in W$ for any point $(w_1,w_3)\in U_2,$ and $\sigma_i(w_1,w_3,\xi(w_1,w_3))\neq0$ for any point $(w_1,w_3)\in U_2$ and $i=1,5$. On $U_2$ the function $\xi(w_1,w_3)$ is expressed as
\begin{gather*}
\xi(w_1,w_3)=\frac{w_3^2}{w_1}+\frac{1}{3}w_1^2w_3-\frac{1}{45}w_1^5.
\end{gather*}
We have
\begin{gather*}
U(x,t)=4F_2(x,t,\xi(x,t))=-2\frac{\sigma_3(x,t,\xi(x,t))}{\sigma_1(x,t,\xi(x,t))}=\frac{6x\big(x^3+6t\big)}{\big(x^3-3t\big)^2},\\
V(x,t)=2\big\{F_4(x,t,\xi(x,t))-F_2(x,t,\xi(x,t))^2\big\}=-2\frac{\sigma_5(x,t,\xi(x,t))}{\sigma_1(x,t,\xi(x,t))}=-\frac{18x^2}{\big(x^3-3t\big)^2}.
\end{gather*}

\begin{Theorem}The function $U(x,t)$ is a solution of the KdV-hierarchy
\begin{gather}
U'''-4\dot{U}-6UU'=0,\label{kdvequation} \\
\dot{U}''-4U\dot{U}-2U'V=0,\qquad \dot{U}=V'.\nonumber
\end{gather}
\end{Theorem}

\begin{proof}This theorem follows from Theorem \ref{dkdvs}.
\end{proof}

Let us consider the curve $V_2$ of genus 2. It is well known that the function $D(x,t)=2\wp_{1,1}(x,t)$ is a solution of the KdV-hierarchy (see \cite[Theorems~5.1 and~5.2]{B1}, \cite[Theorem~3.6]{B2}, \cite[Theorem~6]{M1})
\begin{gather}
D'''-4\dot{D}-6DD'=0,\qquad \dot{D}''-4D\dot{D}-2D'E=0,\qquad \dot{D}=E'.\label{rkdvh}
\end{gather}
where $E(x,t)=2\wp_{1,3}(x,t)$. Let the constant vector $(y_4,y_6,y_8,y_{10})\in\mathbb{C}^4$ tends to zero. Then we have
\begin{gather*}
\sigma=-w_3+\frac{1}{3}w_1^3,\qquad \sigma_1=w_1^2,\qquad \sigma_3=-1,\qquad \sigma_{11}=2w_1,\qquad \sigma_{13}=0,\\
D(x,t)=2\frac{\sigma_1^2-\sigma_{11}\sigma}{\sigma^2}=\frac{6x\big(x^3+6t\big)}{\big(x^3-3t\big)^2},\qquad E(x,t)=2\frac{\sigma_1\sigma_3-\sigma_{13}\sigma}{\sigma^2}=-\frac{18x^2}{\big(x^3-3t\big)^2},
\end{gather*}
which is a solution of the KdV-hierarchy (\ref{rkdvh}). Note that $U(x,t)=D(x,t)$ and $V(x,t)=E(x,t)$ in the rational limit.

\begin{Remark}The Lax form of the KdV-equation (\ref{kdvequation}) is
\begin{gather*}\frac{{\rm d} L}{{\rm d}t}=[A,L],\end{gather*}
where
\begin{gather*}L=-\partial_x^2+U,\qquad A=\partial_x^3-\frac{3}{2}U\partial_x-\frac{3}{4}U_x.\end{gather*}
The Schr\"odinger equation with the potential $U(x,t)=6x\big(x^3+6t\big)/\big(x^3-3t\big)^2$ is
\begin{gather}
-\frac{{\rm d}\eta}{{\rm d}x^2}+U(x,t)\eta(x,t)=\lambda\eta(x,t),\label{schure}
\end{gather}
where $\eta(x,t)$ is an unknown function and $\lambda\in\mathbb{C}$. At each fixed time $t$, the function $U(x,t)=6x\big(x^3+6t\big)/\big(x^3-3t\big)^2$ decreases as $O\big(1/x^2\big)$ for $x \to \infty$. Thus, we obtain a~meaningful solution of the corresponding Schr\"odinger equation~(\ref{schure}) in the point of view of physics.
\end{Remark}

\appendix

\section{Correction of Section 8 in \cite{AB}}\label{appendixA}

In \cite[Section 8]{AB}, we derived the solution of the dynamical systems introduced in~\cite{BM} in the case of $(y_4,\dots,y_{14})=(0,\dots,0)$ for $g=3$. Unfortunately, there are miscalculations in the expressions of $F_5$, $F_7$ and Examples~1 and~3. In this appendix we correct the errors.

We assume $g=3$ and consider the curve $V_3$. Let the constant vector $(y_4,\dots,y_{14})\in\mathbb{C}^6$ tend to zero. Then the sigma function $\sigma(w_1,w_3,w_5)$ transforms into the Schur--Weierstrass polynomial as Section~\ref{section8}. Take a point $w^{(0)}=\big(w_1^{(0)},w_3^{(0)},w_5^{(0)}\big)\in W$ such that $\sigma_1(w^{(0)})\neq0$. In a~sufficiently small open neighborhood $U_3\subset\mathbb{C}^2$ of $\big(w_3^{(0)},w_5^{(0)}\big)$, there exists a uniquely determined holomorphic function $\varphi(w_3,w_5)$ on $U_3$ such that $\varphi\big(w_3^{(0)},w_5^{(0)}\big)=w_1^{(0)}$, $(\varphi(w_3,w_5),w_3,w_5)\in W$ for any point $(w_3,w_5)\in U_3,$ and $\sigma_1(\varphi(w_3,w_5),w_3,w_5)\neq0$ for any point $(w_3,w_5)\in U_3$. We have the relation
\begin{gather}
\varphi(w_3,w_5)^6=15\varphi(w_3,w_5)^3w_3+45w_3^2-45\varphi(w_3,w_5)w_5.\label{rel}
\end{gather}
The definitions of $F_2$, $F_4$, $F_5$, $F_7$ (see Section~\ref{section4}) and the relation~(\ref{rel}) imply\footnote{We used the computer algebra system Maxima for calculation.}
\begin{gather*}
F_2(\varphi(w_3,w_5),w_3,w_5)=5\big(\varphi^3+6w_3\big)/K_2, \\
F_4(\varphi(w_3,w_5),w_3,w_5)=15\big({-}\varphi^3w_3+15\varphi w_5-15w_3^2\big)/K_4, \\
F_5(\varphi(w_3,w_5),w_3,w_5)=\big({-}15 {{{w_5}}^{2}}-195 {{{\varphi}}^{2}} {w_3} {w_5}+8 {{{\varphi}}^{5}} {w_5}+135 {\varphi} {{{w_3}}^{3}}+63 {{{\varphi}}^{4}} {{{w_3}}^{2}}\big)/K_5, \\
F_7(\varphi(w_3,w_5),w_3,w_5)=-15 {\varphi}\big( 25 {{{\varphi}}^{2}} {{{w_5}}^{2}}\!-45 {\varphi} {{{w_3}}^{2}} {w_5}-15 {{{\varphi}}^{4}} {w_3} {w_5}+27 {{{w_3}}^{4}}\!+18 {{{\varphi}}^{3}} {{{w_3}}^{3}}\big)\!/K_7,
\end{gather*}
where
\begin{gather*}
K_2=2\big(2\varphi^5-15\varphi^2w_3+15w_5\big),\qquad K_4=4\big({-}8\varphi^5w_5+27\varphi^4w_3^2-30\varphi^2w_3w_5+15w_5^2\big),\nonumber\\
K_5=3\big(5 {{{w_5}}^{3}}+165 {{{\varphi}}^{2}} {w_3} {{{w_5}}^{2}}+14 {{{\varphi}}^{5}} {{{w_5}}^{2}}-585 {\varphi} {{{w_3}}^{3}} {w_5}-111 {{{\varphi}}^{4}} {{{w_3}}^{2}} {w_5}\\
\hphantom{K_5=}{} +405 {{{w_3}}^{5}}+189 {{{\varphi}}^{3}} {{{w_3}}^{4}}\big), \\
K_7=2\big(15 {{{w_5}}^{4}}-4380 {{{\varphi}}^{2}} {w_3} {{{w_5}}^{3}}-208 {{{\varphi}}^{5}} {{{w_5}}^{3}}+28620 {\varphi} {{{w_3}}^{3}} {{{w_5}}^{2}}+3042 {{{\varphi}}^{4}} {{{w_3}}^{2}} {{{w_5}}^{2}} \\
\hphantom{K_7=}{}-24300 {{{w_3}}^{5}} {w_5}-11583 {{{\varphi}}^{3}} {{{w_3}}^{4}} {w_5}+2187 {{{\varphi}}^{2}} {{{w_3}}^{6}}+729 {{{\varphi}}^{5}} {{{w_3}}^{5}}\big).
\end{gather*}
The dynamical systems introduced in \cite{BM} for $g=3$ and $y_4=y_6=y_8=y_{10}=0$ are as follows
\begin{alignat*}{3}
& (I) \quad && \partial_tG_2=-G_5,\qquad \partial_t G_4=-2G_7,\qquad \partial_tG_5=-35G_2^4-42G_2^2G_4-3G_4^2,&\\
&&& \partial_tG_7=-7\big(3G_2^5+10G_2^3G_4+3G_2G_4^2\big),& \\
& (II) \quad && \partial_{\tau} G_2=G_2G_5-G_7,\qquad \partial_{\tau} G_4=2(G_2G_7-G_4G_5),& \\
&&& \partial_{\tau}G_5=G_5^2+14G_2^5-28G_2^3G_4-18G_2G_4^2,& \\
&&& \partial_{\tau}G_7=-G_5G_7+21G_2^6+35G_2^4G_4-21G_2^2G_4^2-3G_4^3.&
\end{alignat*}

\begin{Theorem}[{\cite[Theorem 8.1]{AB}}]\label{ratc} The set of functions $(G_2,G_4,G_5,G_7)$, where
\begin{gather*}
G_i(t,\tau)=F_i(\varphi(t,\tau),t,\tau),\qquad i=2,4,5,7,
\end{gather*}
is a solution of the dynamical systems {\rm (I)} and {\rm (II)}.
\end{Theorem}

\noindent{\bf Example 1.} Let $w^{(0)}=(0,0,1)$. Then $w^{(0)}\in W$ and $\sigma_1\big(w^{(0)}\big)\neq0$. The function $\varphi(t,1)$ of $t$ has the following expansion in a neighborhood of the point $t=0$:
\begin{gather*}\varphi(t,1)=t^2+\frac{1}{3}t^7+\frac{14}{45}t^{12}+\cdots.\end{gather*}
According to Theorem \ref{ratc}, the set of functions $(G_2,G_4,G_5,G_7)$, where
\begin{gather*}
G_i(t)=F_i(\varphi(t,1),t,1),\qquad i=2,4,5,7,
\end{gather*}
is a solution of the dynamical system~(I).

\medskip
\noindent{\bf Example 3.} Let $w^{(0)}=(q,1,0)$ such that $q^6=15q^3+45$. Then $w^{(0)}\in W$, $\sigma_1\big(w^{(0)}\big)={q}^{2}\big(2{q}^{3}-15\big)/15\neq0$, and $\varphi(t,0)=qt^{1/3}$ around $t=1$. We have
\begin{gather*}
F_2\big(qt^{1/3},t,0\big)=\frac{q}{6}t^{-2/3},\qquad F_4\big(qt^{1/3},t,0\big)=-\frac{5\big(q^3+15\big)}{36q^4}t^{-4/3}, \\
F_5\big(qt^{1/3},t,0\big)=\frac{q}{9}t^{-5/3},\qquad F_7\big(qt^{1/3},t,0\big)=-\frac{5(2q^3+3)}{54q\big(q^3+3\big)}t^{-7/3}.
\end{gather*}
According to Theorem~\ref{ratc}, the set of functions $(G_2,G_4,G_5,G_7)$, where
\begin{gather*}
G_i(t)=F_i\big(qt^{1/3},t,0\big),\qquad i=2,4,5,7,
\end{gather*}
is a solution of the dynamical system (I).

\subsection*{Acknowledgements}

The authors are grateful to Shigeki Matsutani for telling the relations between his paper \cite{M2} and our results in this paper, and for valuable comments.

\pdfbookmark[1]{References}{ref}
\LastPageEnding

\end{document}